\documentclass[11pt]{amsart}

\usepackage{amsmath}
\usepackage{amssymb}
\usepackage{amsthm}
\usepackage{booktabs}
\usepackage[margin=1in]{geometry}
\usepackage{hyperref}

\theoremstyle{plain}
\newtheorem{theorem}{Theorem}[section]
\newtheorem{lemma}[theorem]{Lemma}
\newtheorem{corollary}[theorem]{Corollary}
\newtheorem{proposition}[theorem]{Proposition}

\theoremstyle{remark}
\newtheorem*{remark}{Remark}

\newcommand{\Z}{\mathbb{Z}}
\newcommand{\F}{\mathbb{F}}

\begin{document}

\title[Polychromatic blocks and van der Waerden numbers]{Polychromatic blocks and explicit lower bounds for multicolour van der
Waerden numbers}

\author{Seonbin Kim}
\address{Independent researcher}
\email{sean.kim.business@gmail.com}
\date{}

\begin{abstract}
Let $W(r,k)$ be the least $N$ such that every $r$-colouring of
$\{1,\dots,N\}$ contains a monochromatic $k$-term arithmetic progression. Many
of the best lower bounds for $W(r,k)$ with $r\ge3$ come from blow-up
recurrences of Blankenship, Cummings and Taranchuk and of Xu, in which a
known colouring is lengthened by a periodic block. Call an $h$-colouring of
$\mathbb{Z}_q$ $k$-polychromatic if every $k$-term progression with nonzero
common difference meets each of the $h$ colours. The argument behind these
recurrences shows that such a colouring gives
$W(r,k)>q\,(W(r-\lceil r/h\rceil,k)-1)$ for all $r\ge h$. We show that
cubic, quartic, quintic and sextic residue colourings give $k$-polychromatic
blocks with $q$ much larger than $k$ (e.g.\ $q=1{,}376{,}317$ for
$h=3$ and $k=28$), and use them to get explicit lower bounds. For three
colours and $9\le k\le28$ we improve on every published bound we know of, among them
$W(3,9)>4{,}002{,}705$ and $W(3,28)>159{,}594{,}201{,}561{,}100{,}726$. For
$3\le r\le6$ and $9\le k\le28$, $77$ of the $80$ cases improve on what the
published bounds give under these recurrences and the other three tie. For $r=5$ and $12\le k\le15$ the best values come from palettes that
are cyclic intervals of colours indexed by quintic residues. We checked every block by exhaustive enumeration and
every base colouring by computer. We also prove that the block modulus is at most
$(W(2,k)-1)/(k-1)$, so the method does not improve the asymptotics. For large $k$ the largest cubic block modulus lies between $k^{4-\varepsilon}$ and
$k^2 4^k$.
\end{abstract}

\maketitle

\section{Introduction}

For integers $r,k \ge 1$ let $W(r,k)$ denote the least $N$ such that every
$r$-colouring of $\{1,\dots,N\}$ contains a monochromatic $k$-term arithmetic
progression. We put the number of colours first, whereas Monroe~\cite{Mon} writes
$W(k,r)$ for the same quantity. All our bounds are strict in the
following sense: an assertion $W(r,k) > N$ means that we exhibit an explicit
$r$-colouring of $N$ consecutive integers with no monochromatic $k$-term
arithmetic progression, so that $N < W(r,k)$.

Exact values of $W(r,k)$ are known only in a few cases, and nearly
all large lower bounds come from algebraic colourings. Rabung's power-residue
method~\cite{Rab} colours the nonzero residues modulo a prime $p$ by a
multiplicative character and checks a short combinatorial condition to rule out
monochromatic progressions. Rabung and
Lotts~\cite{RL} refined and automated the method, building on the cyclic ``zipper'' colourings of Herwig, Heule,
van Lambalgen and van Maaren~\cite{HHLM}. Monroe~\cite{Mon} ran a
distributed search for such colourings and produced two-colour tables up to $k=25$
and three-colour tables up to $k=17$.

Bounds can also be pushed up by \emph{blowing up} a known
colouring by a periodic block. The recurrence of Blankenship, Cummings and Taranchuk
(BCT)~\cite{BCT}
\[
  W(r,k) > p\,\bigl(W(r-\lceil r/p\rceil, k) - 1\bigr),
  \qquad p \le k \text{ prime},
\]
blows up by a block of length $p$. In the remark following their Theorem~2.1 the
authors note that a larger prime $p$ is admissible provided
$p - \lfloor p/r\rfloor < k$. Xu's product construction~\cite{Xu}
combines a cyclic colouring of $\Z_p$ with a colouring of an interval in the same way. Most recently, Pouly~\cite{Pou} gave direct ternary Rabung
certificates and a recurrence closure for $12 \le k \le 28$. These are the
strongest published three-colour bounds we know of. They include
$W(3,27) > 2{,}508{,}909{,}833{,}037$, and Pouly also gives six-colour products such as
$W(6,20)>9{,}040{,}715{,}030{,}019{,}241{,}602$.

\subsection*{This note}
A block on $\Z_q$ has to be an $h$-colouring in which every
$k$-progression of nonzero difference meets each of the $h$ colours. We call such a
colouring \emph{$k$-polychromatic}. Any such block
gives the palette recurrence of Theorem~\ref{thm:palette}, whose three-colour
specialisation (Corollary~\ref{cor:block} with $r=h=3$) is
\[
  W(3,k) > q\,(W(2,k)-1).
\]
Power-residue colourings are $k$-polychromatic for $q$ far larger than
$k$ (up to about $1.38\times 10^{6}$ in our search), whereas the blocks
of~\cite{BCT} require $q-\lfloor q/3\rfloor<k$ and~\cite{Pou} uses only primes
$q\le k$. With the best known lower-order inputs this improves every
published three-colour bound we know of (those of~\cite{Mon,Pou}) for $9 \le k \le 28$.
The gain is a factor of about $2.4$ (at $k=10$) up to about $3.7\times 10^{4}$ (at
$k=28$, measured against the best bound derivable from published inputs, see
Section~\ref{sec:results}). Most multicolour cells for
$r=4,5,6$ improve too (Section~\ref{sec:results}). For $W(5,k)$ with
$12\le k\le15$ we use palettes that are cyclic intervals of colours
(Lemma~\ref{lem:arc}). All constructions are finite and
explicit. In Section~\ref{sec:limits} we bound the block modulus
in terms of $W(2,k)$, so no asymptotic growth rate changes. We also give
character-sum bounds for the size of the cubic blocks.

\section{A polychromatic blow-up}\label{sec:blowup}

Throughout, a \emph{$k$-AP} in $\Z_q$ is a sequence $a, a+d, \dots, a+(k-1)d$
with $d \ne 0$ in $\Z_q$, and $d$ is called its common difference. Write
$[r]=\{1,\dots,r\}$. A map $m\colon \Z_q \to [h]$ is \emph{$k$-polychromatic} if
every $k$-AP in $\Z_q$ meets each of the $h$ classes $m^{-1}(1),\dots,m^{-1}(h)$.

The following is not new. The two-case argument below is the one behind BCT's
Lemma~2.1 and Theorem~2.1~\cite{BCT} and Xu's product~\cite{Xu}. In group form it
is the case $\varepsilon=0$ of the shifted product construction of Fox and
Hunter~\cite[Lemma~4.3]{FH}, and it follows from Hunter's fibre
lemma~\cite[Lemma~3.2]{Hun} applied to each colour class. Colourings by
$s$-subsets of the colour set, as below, were studied for van der Waerden
numbers by Xiu, Li, Liang and Xu~\cite{XLLX}. We state the interval version in a
palette form convenient for the tables and include the short proof.

\begin{theorem}[palette form]\label{thm:palette}
Let $q,n \ge 1$ and $r \ge s \ge 1$. Let $S$ assign to each $u \in \Z_q$ an
$s$-element subset $S(u) \subseteq [r]$, and for each colour $c \in [r]$ put
\[
  A_c = \{\, u \in \Z_q : c \in S(u) \,\}.
\]
Assume
\begin{itemize}
  \item[$(\star)$] for every $c \in [r]$, the set $A_c$ contains no $k$-AP, i.e.\
  no progression $u, u+d, \dots, u+(k-1)d$ with $d \not\equiv 0 \pmod q$ lies
  entirely in $A_c$.
\end{itemize}
Let $\beta\colon \{0,\dots,n-1\} \to [s]$ be an $s$-colouring with no
monochromatic $k$-AP, and for each $u$ fix a bijection $\phi_u\colon [s] \to
S(u)$. Define $\chi\colon \{0,\dots,qn-1\} \to [r]$ by
\[
  \chi(qt+u) = \phi_u\bigl(\beta(t)\bigr), \qquad 0 \le u < q,\ 0 \le t < n .
\]
Then $\chi$ has no monochromatic $k$-AP. This gives $W(r,k) > qn$, and taking
$n = W(s,k)-1$ with $\beta$ an extremal $s$-colouring gives
\[
  W(r,k) > q\,\bigl(W(s,k)-1\bigr).
\]
\end{theorem}

\begin{proof}
Suppose for contradiction that $x_0,\dots,x_{k-1}$ with $x_i = x_0 + iD$,
$D \ge 1$, is a monochromatic $k$-AP inside $\{0,\dots,qn-1\}$, of common colour
$c$. Write $r_i = x_i \bmod q$ and $t_i = \lfloor x_i/q\rfloor$ for the residue
and the block index of each term, so that
$\chi(x_i) = \phi_{r_i}(\beta(t_i)) \in S(r_i)$.

\emph{Case $D \not\equiv 0 \pmod q$.} Since $\chi(x_i) = c$ for all $i$ and
$\chi(x_i) \in S(r_i)$, we have $c \in S(r_i)$, i.e.\ $r_i \in A_c$, for every
$i$. But $r_i = (x_0 + iD) \bmod q$ is a $k$-AP in $\Z_q$ with nonzero common
difference $D \bmod q$, so it lies entirely in $A_c$, contradicting $(\star)$.

\emph{Case $D \equiv 0 \pmod q$,} say $D = qD'$ with $D' \ge 1$. Then all terms
share one residue $u := r_0 = \dots = r_{k-1}$, and $\phi_u$ is a fixed
injection. From $\chi(x_i) = \phi_u(\beta(t_i)) = c$ and the injectivity of
$\phi_u$ we get $\beta(t_i) = \phi_u^{-1}(c)$ constant. Now
$t_i = \lfloor x_i/q\rfloor = t_0 + iD'$, so $t_0,\dots,t_0+(k-1)D'$ is a $k$-AP
of common difference $D'$ inside $\{0,\dots,n-1\}$ on which $\beta$ is constant,
contradicting the choice of $\beta$.

Each case gives a contradiction, which shows that $\chi$ has no monochromatic $k$-AP
and that $qn < W(r,k)$. Choosing $\beta$ extremal of length $n = W(s,k)-1$ gives
$W(r,k) > q(W(s,k)-1)$.
\end{proof}

\begin{corollary}[polychromatic block]\label{cor:block}
Let $m\colon \Z_q \to [h]$ be $k$-polychromatic. Then for every $r \ge h$,
\[
  W(r,k) > q\,\bigl(W(r-\lceil r/h\rceil,\,k)-1\bigr).
\]
\end{corollary}

\begin{proof}
Partition $[r]$ into $h$ blocks $B_1,\dots,B_h$ with $|B_j| \le \lceil
r/h\rceil$, and set $s = r - \lceil r/h\rceil$. For each $u \in \Z_q$ choose any
$s$-subset $S(u) \subseteq [r] \setminus B_{m(u)}$, which exists because
$\lvert [r]\setminus B_{m(u)}\rvert = r - |B_{m(u)}| \ge r - \lceil r/h\rceil =
s$. If $c \in B_j$, then $c \in S(u)$ forces $c \notin B_{m(u)}$, that is,
$m(u) \ne j$. This gives
\[
  A_c \subseteq \{\, u : m(u) \ne j \,\} = \Z_q \setminus m^{-1}(j).
\]
Because $m$ is $k$-polychromatic, every $k$-AP meets $m^{-1}(j)$, so no $k$-AP
lies in $A_c$. Condition $(\star)$ holds and Theorem~\ref{thm:palette} applies.
\end{proof}

Taking $h=r$ (so $\lceil r/h\rceil = 1$ and $s = r-1$) recovers the
colour-stripping recurrence $W(r,k) > q(W(r-1,k)-1)$ from any $r$-polychromatic
block on $\Z_q$.

\begin{remark}[the $r=3$ mod-$3$ block]
For $r=h=3$ we have $s=2$ and singleton blocks $B_j=\{j\}$, so that
$S(u)=\{1,2,3\}\setminus\{m(u)\}$. Identifying colours with $\Z/3$ and taking
$\phi_u(\beta)=\bigl(m(u)+1+\beta\bigr)\bmod 3$ for $\beta\in\{0,1\}$, the
construction is the ``mod-$3$ shift''
$\chi(qt+u)=\bigl(m(u)+1+\beta(t)\bigr)\bmod 3$ used in the three-colour tables
below.
\end{remark}

\subsection*{Special cases}
\begin{itemize}
  \item \emph{BCT.} If $q=h=p$ is a prime with $p\le k$ and $m$ is the identity
  $m(j)=j$, then every $k$-AP of nonzero difference modulo $p$ runs through all
  $p$ residues (its difference is invertible and $k\ge p$) and so meets each of the $p$
  colours. This is BCT's Theorem~2.1~\cite{BCT}.
  \item \emph{Pigeonhole blocks.} More generally, let $q$ be prime and
  $m(j)=j\bmod h$. A $k$-AP of nonzero difference modulo $q$ has $\min(k,q)$
  distinct terms, so if $q-\lfloor q/h\rfloor < k$ no colour class can be
  avoided and $m$ is $k$-polychromatic. Here $q$ may exceed $k$ (e.g.\ $q=37$,
  $h=3$, $k=26$, where $37-\lfloor 37/3\rfloor = 25 < 26$), which is the
  admissible-prime remark following~\cite[Thm.~2.1]{BCT}.
  \item \emph{Xu's product.} As stated by Monroe~\cite[Section~2]{Mon}, Xu's
  construction is: if $p$ is a prime and $\Z_p$ has an $a$-colouring $m$ with no
  monochromatic nontrivial $k$-AP, then $W(at,k) > p\,(W(t,k)-1)+1$ (colours
  first). Theorem~\ref{thm:palette} with $r=at$, palettes of size $t$ and
  disjoint palettes $S(u)=B_{m(u)}$, $|B_j|=t$, recovers this recurrence up to
  the additive $1$, and Corollary~\ref{cor:block} with $h=2$ and $r$ even is
  its case $a=2$.
\end{itemize}

\subsection*{Relation to prior work}
The two-case argument above underlies BCT's Lemma~2.1 and
Theorem~2.1~\cite{BCT}, Xu's product~\cite{Xu}, Hunter's Lemma~3.2 (AP-free
fibres over AP-free support)~\cite{Hun} and the cyclic unrolling in
Lemma~2 of Campos, Fox and Schildkraut~\cite{CFS}. Lemmas~4.2 and~4.3 of
Fox and Hunter~\cite{FH} are randomised versions in general groups that tolerate a small error, and
their Corollary~3.3 gives polychromatic colourings of exponential size by the
local lemma. Fox and Hunter also suggest that their methods might give new
constructions for small van der Waerden numbers with at least three colours.
New here are the explicit $h\ge 3$ power-residue polychromatic blocks with
$q$ much larger than $k$, the tables obtained from them with their verification, and the
bounds of Section~\ref{sec:limits}. We are not aware of earlier use of
$h\ge 3$ polychromatic blocks with $q>k$ for explicit van der Waerden lower
bounds.

\section{Power-residue blocks}\label{sec:residue}

\begin{lemma}\label{lem:window}
Let $q$ be a prime, $e \mid q-1$, and fix a primitive root $g$ of $\F_q$. Let
$\psi\colon \F_q^{*} \to \Z_e$ be the index modulo $e$, i.e.\ $\psi(g^a)=a\bmod
e$, so that $\psi(xy)=\psi(x)+\psi(y)$ in $\Z_e$. Let $\sigma\colon \Z_e \to [h]$
be surjective and fix $z \in [h]$. Define $m\colon \Z_q \to [h]$ by
$m(y)=\sigma(\psi(y))$ for $y\ne 0$ and $m(0)=z$. For $s \in \Z_e$ set
$f_s(y)=\sigma(s+\psi(y))$ for $y\ne 0$ and $f_s(0)=z$. Then $m$ is
$k$-polychromatic if and only if, for every $s \in \Z_e$, every cyclic window
$\{u,u+1,\dots,u+k-1\} \pmod q$ meets each of the $h$ colours under $f_s$.
\end{lemma}

\begin{proof}
Fix a $k$-AP $a, a+d, \dots, a+(k-1)d$ in $\Z_q$ with $d \ne 0$, and set
$u = a\,d^{-1}$, so $a+id = d(u+i)$. If $u+i \ne 0$ then by multiplicativity
\[
  m(a+id) = \sigma\bigl(\psi(d)+\psi(u+i)\bigr) = f_s(u+i), \qquad s = \psi(d),
\]
while if $u+i=0$ then $a+id=0$ and $m(a+id)=z=f_s(0)$. The colour sequence
of the progression is then $f_s(u),f_s(u+1),\dots,f_s(u+k-1)$, read along the
consecutive residues $u,\dots,u+k-1$. As $d$ ranges over $\F_q^{*}$ the index
$s=\psi(d)$ ranges over all of $\Z_e$ and $u=ad^{-1}$ ranges over all of $\Z_q$.
So every $k$-AP meets each of the $h$ colours if and only if, for each $s \in \Z_e$,
every length-$k$ cyclic window meets each of the $h$ colours under $f_s$.
\end{proof}

The criterion is an $O(eq)$ test: for each of the $e$ shifts $s$ one scans the
residues $0,1,\dots,q-1$ once and keeps track of the colours present in the current
length-$k$ window. It is also monotone in $k$: a $k$-polychromatic colouring is
$(k+1)$-polychromatic, because a $(k+1)$-AP contains a $k$-AP.

The tables of Section~\ref{sec:results} use two parameter regimes. The simplest
is \emph{$e=h$ with $\sigma=\mathrm{id}$}: the colour of $y\ne 0$ is just its
index modulo $h$. For $h=3$ this is the cubic-residue colouring
\[
  m(y)=\begin{cases}
    0 & y^{(q-1)/3}=1,\\
    1 & y^{(q-1)/3}=\omega,\\
    2 & y^{(q-1)/3}=\omega^2,
  \end{cases}
  \qquad m(0)=0,
\]
with $\omega$ a primitive cube root of unity, and it supplies all the $h=3$
blocks in Tables~\ref{tab:main}, \ref{tab:multi4} and~\ref{tab:multi56}; the
analogous quartic, quintic and sextic colourings supply the $h=4,5,6$ blocks.
The second regime \emph{groups} a finer index class into fewer colours via a
non-injective $\sigma$. We searched $e=h$ for $3\le h\le 6$ together with the
groupings $e=6\to h=3$, $e=8\to h=4$ and $e=10\to h=5$. The
groupings never exceeded the plain power-residue maxima, and the tables use
$\sigma=\mathrm{id}$.

\subsection*{Arc palettes}
Theorem~\ref{thm:palette} does not require the palettes to come from a
partition of the colours. Identify the colours with $\Z_r$ and write
$I_s(c)=\{c,c+1,\dots,c+s-1\}\subseteq\Z_r$ for the cyclic interval of length $s$
starting at $c$.

\begin{lemma}\label{lem:arc}
Let $q$ be a prime with $r\mid q-1$, let $1\le s<r$, fix a primitive root $g$
of $\F_q$ and a unit $a$ of $\Z_r$, and let $\psi(g^i)=ai\bmod r$. Fix
$z\in\Z_r$ and put $S(y)=I_s(\psi(y))$ for $y\ne0$ and $S(0)=I_s(z)$. Then
condition $(\star)$ of Theorem~\ref{thm:palette} holds if and only if
\begin{enumerate}
\item[(i)] for every window $W=\{u,\dots,u+k-1\}$ of $k$ consecutive residues
not containing $0$, the set $\psi(W)$ is not contained in any cyclic interval of
length $s$; and
\item[(ii)] for every window $W$ of $k$ consecutive residues containing $0$ and
every $t\in\Z_r$, there is no $c\in S(0)$ with $c\in t+S(y)$ for all
$y\in W\setminus\{0\}$.
\end{enumerate}
In that case $W(r,k)>q\,(W(s,k)-1)$.
\end{lemma}

\begin{proof}
As in Lemma~\ref{lem:window}, write a $k$-AP as $a'+id=d(u+i)$ with $d\ne0$
and $u=a'd^{-1}$, and put $t=\psi(d)$. For $u+i\ne0$ we have
$S(a'+id)=t+S(u+i)$, while the term with $u+i=0$ is $0$ and has palette $S(0)$.
So a colour $c$ lies in every palette along the progression if and only if
$c\in t+S(y)$ for all nonzero $y$ in the window $W=\{u,\dots,u+k-1\}$, and also
$c\in S(0)$ when $0\in W$. Now $c\in t+I_s(\psi(y))$ means
$\psi(y)\in I_s(c-t-s+1)$. If $0\notin W$, a common colour exists for some
$c$ exactly when $\psi(W)$ lies in some interval of length $s$, independently
of $t$. If $0\in W$, the condition is exactly (ii). Since $t=\psi(d)$ takes
every value of $\Z_r$ and $u$ every value of $\Z_q$ as $(a',d)$ varies, the two
conditions together are equivalent to $(\star)$. The bound is then
Theorem~\ref{thm:palette}.
\end{proof}

For $s=r-1$ condition (i) says that every such window meets each of the $r$ index
classes, which is the case $e=h=r$ of Lemma~\ref{lem:window}. For $r=5$ and
$s=3$ the lemma blows up a three-colouring into five colours using a block of
modulus much larger than the quintic blocks of the same $k$. A scan of all
primes $q\le10^5$ with $q\equiv1\pmod5$, $a\in\{1,2\}$ and every $z$ gives the
largest moduli $q=1931$, $2861$, $6131$, $9791$ for $k=12,13,14,15$, with
$(g,a,z)=(2,2,0)$, $(2,1,0)$, $(2,1,0)$, $(11,2,0)$ respectively. Together with
the three-colour lengths of Table~\ref{tab:main} this gives
\begin{align*}
W(5,12)&>1{,}340{,}686{,}056{,}819, & W(5,13)&>7{,}362{,}778{,}358{,}783,\\
W(5,14)&>48{,}771{,}558{,}421{,}350, & W(5,15)&>350{,}569{,}892{,}496{,}577.
\end{align*}
Partial searches for other cells found nothing further. The cells tried were $(r,s)=(4,2)$ for $k\le13$, $(5,3)$
for $16\le k\le20$ and $(6,4)$ for $16\le k\le25$, over the primes between the
modulus needed to improve on the tables and four (for $(4,2)$, eight) times that
modulus. These searches are not exhaustive.

\section{Base colourings}\label{sec:base}

The blow-up of Section~\ref{sec:blowup} needs a lower-order input $\beta$. The
largest ones come from two classical two-colour families, which we also reuse as
$h=2$ blocks.

\subsection*{Legendre bases}

\begin{lemma}\label{lem:legendre}
Let $p$ be a prime with $k < p$, and let $\left(\tfrac{\cdot}{p}\right)$ denote
the Legendre symbol. Define $\beta$ on $\{0,1,\dots,(k-1)p\}$ by
\[
  \beta(n) =
  \begin{cases}
    \bigl[\left(\tfrac{n}{p}\right) = -1\bigr] & \text{if } p \nmid n,\\[2pt]
    j \bmod 2 & \text{if } n = jp,
  \end{cases}
\]
so that $n$ ranges over $(k-1)p + 1$ integers. Write $\chi = \beta$ and suppose that
\begin{itemize}
  \item[(A)] the longest run of equal values among $\chi(1),\dots,\chi(p-1)$ has
  length at most $k-1$; and
  \item[(B)] letting $L$ be the length of the initial run of $\chi$ starting at
  $1$ and $R$ the length of the final run ending at $p-1$, one has
  $L + R \le k-2$ if $\chi(1) = \chi(-1)$, and $L \le k-2$, $R \le k-2$
  otherwise.
\end{itemize}
Then $\beta$ has no monochromatic $k$-AP, so $W(2,k) > (k-1)p + 1$.
\end{lemma}

\begin{proof}
Suppose $a, a+d, \dots, a+(k-1)d$ is a monochromatic $k$-AP in
$\{0,\dots,(k-1)p\}$ with $d \ge 1$.

\emph{Case $d = p$.} The $k$ terms lie in $\{0,\dots,(k-1)p\}$ and are spaced by
$p$, so they are exactly the multiples $0, p, \dots, (k-1)p$. Their colours are
$0,1,0,1,\dots$ by the definition of $\beta$ on multiples of $p$, which alternate
and are not constant for $k \ge 2$.

\emph{Case $d < p$.} (Since $(k-1)d \le (k-1)p$, these two cases are exhaustive.)
Write $\varepsilon(n)=\left(\tfrac{n}{p}\right)\in\{\pm1\}$ for $p\nmid n$, so that
$\beta(n)=\beta(n')$ if and only if $\varepsilon(n)=\varepsilon(n')$. Put
$u \equiv a\,d^{-1} \pmod p$. Then $a+id \equiv d(u+i)$, and for $u+i\not\equiv 0$
we have $\varepsilon(a+id)=\varepsilon(d)\,\varepsilon(u+i)$. Because $k<p$, at most
one index $i$ satisfies $u+i\equiv 0 \pmod p$.

If no index does, the residues $u,u+1,\dots,u+k-1$ are $k$ consecutive nonzero
residues, i.e.\ a block of $k$ consecutive integers in $\{1,\dots,p-1\}$, on which
$\varepsilon$ is constant. This is a run of length $k$ in $\chi(1),\dots,\chi(p-1)$,
contradicting~(A).

Otherwise let $i_0$ be the unique index with $u+i_0\equiv 0$. The term $a+i_0d$
is a multiple of $p$, and its colour plays no role below. The $i_0$ terms before
it correspond to the residues $-i_0,\dots,-1$ and the $k-1-i_0$ terms after it to
the residues $1,\dots,k-1-i_0$. Constancy of the colour forces $\chi$ to be
constant on $\{p-i_0,\dots,p-1\}$ and on $\{1,\dots,k-1-i_0\}$, so $i_0\le R$ and
$k-1-i_0\le L$, giving $L+R\ge k-1$. If $\chi(1)=\chi(-1)$ this contradicts
$L+R\le k-2$. If $\chi(1)\ne\chi(-1)$, the two parts cannot both be nonempty
(their colours would be determined by $\varepsilon(d)\varepsilon(-1)$ and
$\varepsilon(d)\varepsilon(1)$, which differ). So $i_0\in\{0,k-1\}$, and then
$L\ge k-1$ or $R\ge k-1$, which contradicts $L,R\le k-2$.

Since every case contradicts the hypotheses, $\beta$ has no monochromatic
$k$-AP, and as it colours $(k-1)p+1$ integers, $W(2,k) > (k-1)p+1$.
\end{proof}

Conditions (A) and (B) are Rabung's criterion~\cite{Rab,RL} in a form that is easy
to check, and Pouly~\cite{Pou} proves that they are equivalent to
Rabung's published criterion. They also certify the cyclic
colouring of $\Z_p$ (with either colour assigned to $0$) to be free of
monochromatic nontrivial $k$-APs, so each Legendre prime is also an $h=2$ block
of modulus $p$.

\subsection*{Zipped bases}

\begin{lemma}[cyclic zipper]\label{lem:zipper}
Let $p$ be an odd prime with $p>k$, let $L(y)=1$ iff $y$ is a quadratic
nonresidue modulo $p$, and let $\oplus$ denote addition modulo $2$. Colour
$\{0,1,\dots,2(k-1)p\}$ by
\[
  c(x)=
  \begin{cases}
    L(x\bmod p)\oplus (x\bmod 2) & p\nmid x,\\
    \beta(j) & x=jp,\ 0\le j\le 2k-2,
  \end{cases}
\]
where $\beta$ has no constant run of length $k$ and $\beta(0),\beta(2),\dots,
\beta(2k-2)$ are not all equal (e.g.\ $\beta$ the indicator of $\{k-1\}$ for $k$
odd and of $\{0,k-1\}$ for $k$ even). Suppose, for every cyclic window
$u,\dots,u+k-1$ of $\Z_p$,
\begin{itemize}
  \item[(Z1)] $L$ is non-constant on the nonzero elements of the window, and
  \item[(Z2)] $i\mapsto L(u+i)\oplus i$ is non-constant on the nonzero elements
  of the window.
\end{itemize}
Then $c$ has no monochromatic $k$-AP, so $W(2,k)>2(k-1)p+1$. The cyclic
colouring of $\Z_{2p}$ obtained from the same rule (with $c(0),c(p)$ the two
distinct colours) also has no monochromatic nontrivial $k$-AP, i.e.\ it is a $2$-poly\-chromatic
block of modulus $2p$.
\end{lemma}

\begin{proof}
Let $a,a+d,\dots,a+(k-1)d$ be a $k$-AP with $d\ge1$. Since
$(k-1)d\le 2(k-1)p$, we have $d\le 2p$.

\emph{$p\nmid d$.} The terms are distinct modulo $p$ (as $k<p$), so at most one
is a multiple of $p$; dropping it (its colour is irrelevant) and dilating by
$d^{-1}$ modulo $p$ identifies the nonzero residues with a cyclic window
$u,\dots,u+k-1$, on which $L(a+id)=L(d)\oplus L(u+i)$. If $d$ is even all terms
have one parity, so $c$ is constant iff $L$ is constant on the window's nonzero
elements, contradicting (Z1). If $d$ is odd the parities alternate, so $c$ is
constant iff $i\mapsto L(u+i)\oplus i$ is constant, contradicting (Z2).

\emph{$d=p$.} The non-multiple terms share one residue but alternate in parity.
Their colours then alternate and cannot be constant. If instead every term is a
multiple of $p$, the indices form a run of length $k$ in $\beta$, which is excluded.

\emph{$d=2p$.} Then $(k-1)d=2(k-1)p$ forces $a=0$, so the terms are
$0,2p,\dots,2(k-1)p$, coloured $\beta(0),\beta(2),\dots,\beta(2k-2)$, which are
not all equal.

For the cyclic colouring of $\Z_{2p}$ the same analysis applies with $d\le 2p$,
the case $d=p$ pairing $0$ and $p$ (distinct colours) or opposite parities.
\end{proof}

\subsection*{Bases used}
We take the larger of the two families at each $k$: for $k=9,11,12$ the zipped
bases from primes $p=2579,9697,29033$ (lengths $41{,}265$, $193{,}941$,
$638{,}727$), and for $k=10$ and $k\ge 13$ the Legendre bases (the prime $11{,}497$
gives length $103{,}474$ at $k=10$, and the $k\ge13$ primes are the Monroe/Pouly values).
These zipped primes are the published zipped records, and we recovered them
as the largest primes below $2{,}000{,}000$ passing (Z1) and (Z2). For $13\le k\le
16$ the zipped primes in $[p_0/2,\,2p_0]$ do not improve on the Legendre length (the
largest passing zipped primes are $49{,}739$ for $k=13$ and $112{,}571$ for
$k=14$). We do not claim any new two-colour records. As $h=2$ blocks we use both the
Legendre primes $p_0$ and the zipped moduli $2p$ of Lemma~\ref{lem:zipper}, which are
$5{,}158$, $7{,}754$, $19{,}394$, $58{,}066$, $99{,}478$ and $225{,}142$ for
$k=9,\dots,14$. Each was checked exhaustively.

\section{Results}\label{sec:results}

Table~\ref{tab:main} gives the three-colour certificates. Its columns give the
length $k$, the $h=3$ cubic-residue block prime $q$ of Lemma~\ref{lem:window},
the base length $n$ from Section~\ref{sec:base}, the new bound $W(3,k)>qn$ from
Theorem~\ref{thm:palette} and the best published bound we know of. Some block primes
repeat across consecutive rows (as with $q=6949$ at $k=16,17$ and
$q=576727$ at $k=25,26$) because the window condition for $k$ already certifies
$k+1$ by the monotonicity of Lemma~\ref{lem:window}.

\begin{table}[htbp]
\centering\small
\begin{tabular}{rrrrrl}
\toprule
$k$ & $q$ & base length $n$ & new bound $W(3,k)>qn$ & best published & source \\
\midrule
9 & 97 & 41\,265 & 4\,002\,705 & 932\,745 & M \\
10 & 97 & 103\,474 & 10\,036\,978 & 4\,173\,724 & M \\
11 & 1\,087 & 193\,941 & 210\,813\,867 & 18\,603\,731 & M \\
12 & 1\,087 & 638\,727 & 694\,296\,249 & 79\,134\,144 & M \\
13 & 1\,567 & 1\,642\,309 & 2\,573\,498\,203 & 251\,282\,317 & M \\
14 & 2\,551 & 3\,118\,350 & 7\,954\,910\,850 & 669\,256\,082 & M \\
15 & 4\,201 & 8\,523\,047 & 35\,805\,320\,447 & 2\,250\,960\,279 & M \\
16 & 6\,949 & 16\,370\,086 & 113\,755\,727\,614 & 9\,186\,001\,216 & M \\
17 & 6\,949 & 46\,397\,777 & 322\,418\,152\,373 & 33\,166\,190\,897 & P \\
18 & 17\,791 & 91\,079\,252 & 1\,620\,390\,972\,332 & 100\,289\,348\,022 & P \\
19 & 31\,333 & 250\,546\,915 & 7\,850\,386\,487\,695 & 259\,446\,988\,783 & P \\
20 & 77\,431 & 526\,317\,462 & 40\,753\,287\,400\,122 & 326\,368\,775\,450 & P \\
21 & 77\,431 & 1\,409\,670\,741 & 109\,152\,215\,146\,371 & 343\,546\,079\,421 & P \\
22 & 95\,803 & 2\,582\,037\,634 & 247\,366\,951\,450\,102 & 360\,723\,383\,392 & P \\
23 & 177\,823 & 6\,206\,141\,987 & 1\,103\,594\,786\,554\,301 & 377\,900\,687\,363 & P \\
24 & 246\,187 & 10\,980\,093\,212 & 2\,703\,156\,207\,582\,644 & 1\,082\,646\,556\,499 & P \\
25 & 576\,727 & 27\,333\,622\,969 & 15\,764\,038\,374\,042\,463 & 1\,082\,646\,556\,499 & P \\
26 & 576\,727 & 53\,126\,634\,776 & 30\,639\,564\,694\,458\,152 & 1\,965\,685\,486\,712 & C \\
27 & 1\,154\,653 & 109\,083\,036\,219 & 125\,953\,055\,019\,377\,007 & 4\,036\,072\,340\,103 & C \\
28 & 1\,376\,317 & 115\,957\,444\,078 & 159\,594\,201\,561\,100\,726 & 4\,290\,425\,430\,886 & C \\
\bottomrule
\end{tabular}
\caption{Three-colour certificates $W(3,k)>qn$ from Theorem~\ref{thm:palette}
with a cubic ($h=3$) block of modulus $q$ and a two-colour base of length $n$.
Sources are M, Monroe~\cite{Mon} (arXiv 2016, journal version
2026); P, Pouly~\cite{Pou} (2026); C, the BCT recurrence with $p=37$ applied by us
to published two-colour bounds (see the comparator paragraph).}
\label{tab:main}
\end{table}

Tables~\ref{tab:multi4} and~\ref{tab:multi56} give the strict multicolour
improvements for $r=4,5,6$. Their columns are $r$, $k$, the new bound, the
construction, and the best bound obtained by closing the published inputs. An entry
``$h$, $q$ $\times$ $W(r',k)$'' means Corollary~\ref{cor:block} applied
with an $h$-polychromatic block of modulus $q$ to the length $L$ realising our
$W(r',k)$ bound in the same column $k$, giving $W(r,k)>qL$. The label
``arc, $q$ $\times$ $W(3,k)$'' means Theorem~\ref{thm:palette} with the
$(5,3)$ arc palette of Lemma~\ref{lem:arc}. Of the $80$ cells
with $3\le r\le 6$ and $9\le k\le 28$, exactly $77$ strictly improve the closure
of the published inputs. The remaining cells, $(r,k)=(4,10)$, $(4,13)$ and
$(5,10)$, only tie it and are left out of the tables.

\begin{table}[htbp]
\centering\small
\begin{tabular}{rrlll}
\toprule
$r$ & $k$ & new bound & construction & published (closure) \\
\midrule
4 & 9 & $2.13\cdot10^{8}$ & $h=2$\, $q=5\,158$ $\times$ $W(2,k)$ & $1.40\cdot10^{8}$ \\
4 & 11 & $1.29\cdot10^{10}$ & $h=4$\, $q=61$ $\times$ $W(3,k)$ & $3.46\cdot10^{9}$ \\
4 & 12 & $4.24\cdot10^{10}$ & $h=4$\, $q=61$ $\times$ $W(3,k)$ & $3.71\cdot10^{10}$ \\
4 & 14 & $1.82\cdot10^{12}$ & $h=4$\, $q=229$ $\times$ $W(3,k)$ & $7.48\cdot10^{11}$ \\
4 & 15 & $8.20\cdot10^{12}$ & $h=4$\, $q=229$ $\times$ $W(3,k)$ & $5.19\cdot10^{12}$ \\
4 & 16 & $1.48\cdot10^{14}$ & $h=4$\, $q=1\,301$ $\times$ $W(3,k)$ & $1.79\cdot10^{13}$ \\
4 & 17 & $4.19\cdot10^{14}$ & $h=4$\, $q=1\,301$ $\times$ $W(3,k)$ & $1.35\cdot10^{14}$ \\
4 & 18 & $2.11\cdot10^{15}$ & $h=4$\, $q=1\,301$ $\times$ $W(3,k)$ & $4.88\cdot10^{14}$ \\
4 & 19 & $1.02\cdot10^{16}$ & $h=4$\, $q=1\,301$ $\times$ $W(3,k)$ & $3.49\cdot10^{15}$ \\
4 & 20 & $1.14\cdot10^{17}$ & $h=4$\, $q=2\,789$ $\times$ $W(3,k)$ & $1.46\cdot10^{16}$ \\
4 & 21 & $3.04\cdot10^{17}$ & $h=4$\, $q=2\,789$ $\times$ $W(3,k)$ & $9.94\cdot10^{16}$ \\
4 & 22 & $9.95\cdot10^{17}$ & $h=4$\, $q=4\,021$ $\times$ $W(3,k)$ & $3.17\cdot10^{17}$ \\
4 & 23 & $1.22\cdot10^{19}$ & $h=4$\, $q=11\,093$ $\times$ $W(3,k)$ & $1.75\cdot10^{18}$ \\
4 & 24 & $3.00\cdot10^{19}$ & $h=4$\, $q=11\,093$ $\times$ $W(3,k)$ & $5.24\cdot10^{18}$ \\
4 & 25 & $2.86\cdot10^{20}$ & $h=4$\, $q=18\,149$ $\times$ $W(3,k)$ & $3.11\cdot10^{19}$ \\
4 & 26 & $6.40\cdot10^{20}$ & $h=4$\, $q=20\,873$ $\times$ $W(3,k)$ & $1.13\cdot10^{20}$ \\
4 & 27 & $3.52\cdot10^{21}$ & $h=4$\, $q=27\,917$ $\times$ $W(3,k)$ & $4.58\cdot10^{20}$ \\
4 & 28 & $5.91\cdot10^{21}$ & $h=4$\, $q=37\,049$ $\times$ $W(3,k)$ & $4.98\cdot10^{20}$ \\
\bottomrule
\end{tabular}
\caption{Strict multicolour improvements for $r=4$ from
Corollary~\ref{cor:block}. The last column is the comparator described below
(published inputs from 2016 to 2026, closed under the BCT and Xu recurrences).}
\label{tab:multi4}
\end{table}

\begin{table}[htbp]
\centering\small
\begin{tabular}{rrlll}
\toprule
$r$ & $k$ & new bound & construction & published (closure) \\
\midrule
5 & 9 & $1.49\cdot10^{9}$ & $h=5$\, $q=7$ $\times$ $W(4,k)$ & $9.79\cdot10^{8}$ \\
5 & 11 & $2.29\cdot10^{11}$ & $h=3$\, $q=1\,087$ $\times$ $W(3,k)$ & $3.81\cdot10^{10}$ \\
5 & 12 & $1.34\cdot10^{12}$ & arc, $q=1\,931$ $\times$ $W(3,k)$ & $4.82\cdot10^{11}$ \\
5 & 13 & $7.36\cdot10^{12}$ & arc, $q=2\,861$ $\times$ $W(3,k)$ & $2.92\cdot10^{12}$ \\
5 & 14 & $4.88\cdot10^{13}$ & arc, $q=6\,131$ $\times$ $W(3,k)$ & $9.72\cdot10^{12}$ \\
5 & 15 & $3.51\cdot10^{14}$ & arc, $q=9\,791$ $\times$ $W(3,k)$ & $8.82\cdot10^{13}$ \\
5 & 16 & $9.03\cdot10^{15}$ & $h=5$\, $q=61$ $\times$ $W(4,k)$ & $3.04\cdot10^{14}$ \\
5 & 17 & $2.56\cdot10^{16}$ & $h=5$\, $q=61$ $\times$ $W(4,k)$ & $2.56\cdot10^{15}$ \\
5 & 18 & $4.45\cdot10^{17}$ & $h=5$\, $q=211$ $\times$ $W(4,k)$ & $9.27\cdot10^{15}$ \\
5 & 19 & $2.16\cdot10^{18}$ & $h=5$\, $q=211$ $\times$ $W(4,k)$ & $6.63\cdot10^{16}$ \\
5 & 20 & $7.17\cdot10^{19}$ & $h=5$\, $q=631$ $\times$ $W(4,k)$ & $3.35\cdot10^{17}$ \\
5 & 21 & $1.92\cdot10^{20}$ & $h=5$\, $q=631$ $\times$ $W(4,k)$ & $2.29\cdot10^{18}$ \\
5 & 22 & $6.28\cdot10^{20}$ & $h=5$\, $q=631$ $\times$ $W(4,k)$ & $7.30\cdot10^{18}$ \\
5 & 23 & $8.58\cdot10^{21}$ & $h=5$\, $q=701$ $\times$ $W(4,k)$ & $4.03\cdot10^{19}$ \\
5 & 24 & $2.10\cdot10^{22}$ & $h=5$\, $q=701$ $\times$ $W(4,k)$ & $1.21\cdot10^{20}$ \\
5 & 25 & $2.01\cdot10^{23}$ & $h=5$\, $q=701$ $\times$ $W(4,k)$ & $9.03\cdot10^{20}$ \\
5 & 26 & $1.29\cdot10^{24}$ & $h=5$\, $q=2\,011$ $\times$ $W(4,k)$ & $3.50\cdot10^{21}$ \\
5 & 27 & $1.10\cdot10^{25}$ & $h=5$\, $q=3\,121$ $\times$ $W(4,k)$ & $1.42\cdot10^{22}$ \\
5 & 28 & $1.85\cdot10^{25}$ & $h=5$\, $q=3\,121$ $\times$ $W(4,k)$ & $1.54\cdot10^{22}$ \\
6 & 9 & $2.06\cdot10^{10}$ & $h=2$\, $q=5\,158$ $\times$ $W(3,k)$ & $6.85\cdot10^{9}$ \\
6 & 10 & $1.15\cdot10^{11}$ & $h=2$\, $q=11\,497$ $\times$ $W(3,k)$ & $9.16\cdot10^{10}$ \\
6 & 11 & $1.40\cdot10^{13}$ & $h=3$\, $q=1\,087$ $\times$ $W(4,k)$ & $4.19\cdot10^{11}$ \\
6 & 12 & $4.60\cdot10^{13}$ & $h=3$\, $q=1\,087$ $\times$ $W(4,k)$ & $6.26\cdot10^{12}$ \\
6 & 13 & $3.52\cdot10^{14}$ & $h=2$\, $q=136\,859$ $\times$ $W(3,k)$ & $3.80\cdot10^{13}$ \\
6 & 14 & $4.65\cdot10^{15}$ & $h=3$\, $q=2\,551$ $\times$ $W(4,k)$ & $1.61\cdot10^{14}$ \\
6 & 15 & $3.44\cdot10^{16}$ & $h=3$\, $q=4\,201$ $\times$ $W(4,k)$ & $1.37\cdot10^{15}$ \\
6 & 16 & $1.03\cdot10^{18}$ & $h=3$\, $q=6\,949$ $\times$ $W(4,k)$ & $1.00\cdot10^{16}$ \\
6 & 17 & $2.91\cdot10^{18}$ & $h=3$\, $q=6\,949$ $\times$ $W(4,k)$ & $9.62\cdot10^{16}$ \\
6 & 18 & $3.75\cdot10^{19}$ & $h=3$\, $q=17\,791$ $\times$ $W(4,k)$ & $5.37\cdot10^{17}$ \\
6 & 19 & $3.20\cdot10^{20}$ & $h=3$\, $q=31\,333$ $\times$ $W(4,k)$ & $3.61\cdot10^{18}$ \\
6 & 20 & $8.80\cdot10^{21}$ & $h=3$\, $q=77\,431$ $\times$ $W(4,k)$ & $9.04\cdot10^{18}$ \\
6 & 21 & $2.90\cdot10^{22}$ & $h=6$\, $q=151$ $\times$ $W(5,k)$ & $5.26\cdot10^{19}$ \\
6 & 22 & $9.53\cdot10^{22}$ & $h=3$\, $q=95\,803$ $\times$ $W(4,k)$ & $1.68\cdot10^{20}$ \\
6 & 23 & $2.43\cdot10^{24}$ & $h=6$\, $q=283$ $\times$ $W(5,k)$ & $9.26\cdot10^{20}$ \\
6 & 24 & $7.38\cdot10^{24}$ & $h=3$\, $q=246\,187$ $\times$ $W(4,k)$ & $2.77\cdot10^{21}$ \\
6 & 25 & $1.65\cdot10^{26}$ & $h=3$\, $q=576\,727$ $\times$ $W(4,k)$ & $2.08\cdot10^{22}$ \\
6 & 26 & $1.86\cdot10^{27}$ & $h=6$\, $q=1\,447$ $\times$ $W(5,k)$ & $1.01\cdot10^{23}$ \\
6 & 27 & $1.59\cdot10^{28}$ & $h=6$\, $q=1\,447$ $\times$ $W(5,k)$ & $4.40\cdot10^{23}$ \\
6 & 28 & $2.67\cdot10^{28}$ & $h=6$\, $q=1\,447$ $\times$ $W(5,k)$ & $4.79\cdot10^{23}$ \\
\bottomrule
\end{tabular}
\caption{Strict multicolour improvements for $r=5$ and $r=6$ from
Corollary~\ref{cor:block} and, in the rows marked ``arc'', from
Lemma~\ref{lem:arc}.}
\label{tab:multi56}
\end{table}

\subsection*{Blocks used}
The $h=3$ blocks are the cubic primes $q$ of Table~\ref{tab:main}. For $h=4$ (quartic) the
blocks are the moduli $13$, $29$, $61$, $61$, $61$, $229$, $229$, $1301$, $1301$, $1301$, $1301$, $2789$, $2789$, $4021$, $11093$, $11093$, $18149$, $20873$, $27917$, $37049$ for $k=9,\dots,28$
respectively. For $h=5$ they are $7$, $11$, $11$, $13$, $13$, $13$, $17$, $61$, $61$, $211$, $211$, $631$, $631$, $631$, $701$, $701$, $701$, $2011$, $3121$, $3121$, where the entries at $k=9,12,13,14,15$ are pigeonhole
blocks. For $h=6$ they are $7$, $7$, $11$, $13$, $13$, $19$, $19$, $43$, $43$, $43$, $67$, $67$, $151$, $151$, $283$, $283$, $397$, $1447$, $1447$, $1447$, with a pigeonhole block at $k=11$.

\subsection*{Comparator}
The ``best published'' and ``published (closure)'' columns are our best attempt at a
maximum over the published literature and are not certified. Each entry is the largest published direct value (Monroe's Table~1~\cite{Mon},
2016/2026, for $r\le 4$, $k\le 14$; the BCT figures~\cite{BCT}, 2018, for
$r=5,6$, $k\le 13$; Pouly's three-colour table~\cite{Pou}, 2026, for $r=3$,
$k\ge 17$; and the direct Rabung rows of Pouly's atlas~\cite{PouAtlas}, 2026,
for $r=3,\dots,6$), closed under
the BCT recurrence (any prime $p$ with $p-\lfloor p/r\rfloor<k$, as allowed by
the remark after~\cite[Thm.~2.1]{BCT}), Xu's product (with the published
two-colour moduli $3389,11497,17863,58013,136859,239873$ for $k=9,\dots,14$ and
Legendre primes beyond), and monotonicity in $r$. We derived some of these comparator
values ourselves by closing published inputs under BCT and Xu, so they are
not published numbers. For $r=3$ and $k=26,27,28$ this closure is the BCT remark
with $p=37$ applied to the published $W(2,k)$, giving $1{,}965{,}685{,}486{,}712$,
$4{,}036{,}072{,}340{,}103$ and $4{,}290{,}425{,}430{,}886$. These exceed Pouly's
published values ($1{,}221{,}912{,}599{,}848$, $2{,}508{,}909{,}833{,}037$,
$2{,}667{,}021{,}213{,}794$), which used $p=23\le k$. For $r=6$ and
$k=18,19,20$ the comparator agrees to the displayed precision with Pouly's
six-colour products~\cite{Pou} ($W(6,18)>537{,}310{,}517{,}724{,}727{,}076$,
$W(6,19)>3{,}611{,}313{,}480{,}298{,}317{,}085$,
$W(6,20)>9{,}040{,}715{,}030{,}019{,}241{,}602$). Pouly's report says its
own recurrence closure covers at most four colours. Apart from these
six-colour products we did not find published five- or six-colour values for
$k\ge14$ other than the atlas rows.

\section{Limits and block sizes}\label{sec:limits}

The method can do no better than the numbers it blows up. Here is the precise
statement, in the generality of Theorem~\ref{thm:palette}.

\begin{proposition}\label{prop:limits}
Let $S\colon\Z_q\to\binom{[r]}{s}$ satisfy condition $(\star)$ of
Theorem~\ref{thm:palette}. Then $W(r-s+1,k)>(k-1)q$. As a special case, if
$m\colon\Z_q\to[h]$ is $k$-polychromatic with $h\ge 2$, then
$q\le (W(2,k)-1)/(k-1)$, and the bound $q\,(W(2,k)-1)$ given by
Corollary~\ref{cor:block} for $r=3$ is at most $(W(2,k)-1)^2/(k-1)$.
\end{proposition}

\begin{proof}
Fix $T\subseteq[r]$ with $|T|=r-s+1$. Since $|[r]\setminus T|=s-1$, every
palette $S(u)$ meets $T$, so we can choose $c(u)\in S(u)\cap T$. For each $t\in T$ the
class $c^{-1}(t)$ lies in $A_t$, so the colouring $c\colon\Z_q\to T$ has no
monochromatic nontrivial $k$-AP. Lay it out periodically on
$\{0,\dots,(k-1)q-1\}$. A $k$-AP there with difference $D$ and $q\nmid D$
reduces to a nontrivial $k$-AP of $\Z_q$ and so is not monochromatic, while
$q\mid D$ would force a span $(k-1)D\ge(k-1)q$, which does not fit. This proves
$W(r-s+1,k)>(k-1)q$. For a $k$-polychromatic block, merging the classes into
two nonempty groups gives the case $r=h$, $s=h-1$, so $W(2,k)>(k-1)q$.
\end{proof}

So for three colours the method gives at most about $W(2,k)^2/k$. It
cannot reach the super-exponential growth $W(3,k)>2^{k\log^{*}k/4}$ proved by
Fox and Hunter~\cite{FH}, and with $W(2,k)\ge(1-o(1))\,k\,2^{k-1}$ from
Campos, Fox and Schildkraut~\cite{CFS} it does not change any asymptotic growth rate.
A family of $k$-polychromatic blocks
with $q^{1/k}\to\infty$, for any fixed $h\ge2$, would give
$W(2,k)^{1/k}\to\infty$, a well-known question of Erd\H{o}s
(see~\cite{EP138}). Blocks of size $c^k$ with
$c>3/2$ are not ruled out by the proposition, which only ties super-exponential blocks to that question.

The random colouring and the local lemma give $k$-polychromatic $h$-colourings
of $\Z_q$ for $q$ as large as $(h/(h-1))^{k}$ up to a polynomial factor in $k$
(this is contained in~\cite[Cor.~3.3]{FH}), but not explicitly. For the cubic
blocks of Section~\ref{sec:residue} we have the following explicit bounds.

\begin{proposition}\label{prop:cubic}
Let $q>k$ be a prime with $q\equiv1\pmod 3$, let $\chi$ be a cubic character modulo $q$,
and let $m$ be the cubic-residue colouring of Section~\ref{sec:residue} (with
any value $m(0)$).
\begin{enumerate}
\item[(a)] If $\bigl|\sum_{y\in I}\chi(y)\bigr|<(k-1)/2$ for every set $I$ of
$k$ consecutive residues modulo $q$, then $m$ is $k$-polychromatic. The
hypothesis holds whenever $\sqrt q\,\log q<(k-1)/2$, and for every
$\varepsilon>0$ there is $k_0(\varepsilon)$ such that it holds for all
$k\ge k_0(\varepsilon)$ and all such primes $q\le k^{4-\varepsilon}$.
\item[(b)] If $q\ge k^2 4^k$, then $m$ is not $k$-polychromatic.
\end{enumerate}
It follows that the largest prime $q_3(k)$ for which the cubic block is
$k$-polychromatic satisfies $k^{4-\varepsilon}\ll q_3(k)<k^2 4^k$.
\end{proposition}

\begin{proof}
(a) Suppose $m$ is not $k$-polychromatic. By Lemma~\ref{lem:window} (with
$\sigma=\mathrm{id}$) some shift $f_s$ misses a colour on some window $I$ of
$k$ consecutive residues; then the nonzero elements of $I$, of which there
are at least $k-1$, avoid one cubic class, say $\chi(y)\ne\omega^{j}$ for all
$y\in I\setminus\{0\}$. Each such $y$ has
$\operatorname{Re}\bigl(\omega^{-j}\chi(y)\bigr)=-1/2$. Since also $\chi(0)=0$, this gives
$\bigl|\sum_{y\in I}\chi(y)\bigr|\ge(k-1)/2$. The first sufficient condition
follows from the inequality of P\'olya and Vinogradov~\cite{IK}
$\bigl|\sum_{y\in I}\chi(y)\bigr|<\sqrt q\log q$. Burgess's bound~\cite{IK} gives
$\sum_{y\in I}\chi(y)=o(|I|)$ uniformly for intervals with
$|I|\ge q^{1/4+\delta}$, which yields the second.

(b) For $y\ne0$, $\mathbf 1[\chi(y)\ne\omega^{j}]
=\tfrac13\bigl(2-\omega^{-j}\chi(y)-\omega^{j}\bar\chi(y)\bigr)$. Let $N$ be the
number of $x\in\F_q$ such that $x+1,\dots,x+k$ are all nonzero and none lies
in the class $\chi=\omega^{j}$. Expanding the product over $i=1,\dots,k$ and
summing over all $x$ gives $2^k q$ from the constant term. Every other term is
a coefficient times $\sum_x\chi\bigl(\prod_i(x+i)^{a_i}\bigr)$ with
$a_i\in\{0,1,2\}$ not all zero. The polynomial $\prod_i(x+i)^{a_i}$ is not a
cube, so by Weil's bound~\cite{IK} each such sum has modulus at most $(k-1)\sqrt q$, and
the coefficients have total modulus at most $4^k-2^k$. The $k$ values of $x$
with some $x+i=0$ contribute at most $4^k$ each. Putting these together,
\[
3^kN\ \ge\ 2^kq-(4^k-2^k)(k-1)\sqrt q-k\,4^k .
\]
With $\sqrt q=t\ge k2^k$ the right-hand side is at least
$2^kt^2-(k-1)2^k\,t\cdot 2^k-k4^k\ge k4^k(2^k-1)>0$. So some window
$x+1,\dots,x+k$ of nonzero residues avoids the class $\omega^{j}$, and the
progression $x+1,\dots,x+k$ of difference $1$ misses the colour of that class.
\end{proof}

The observed values in Table~\ref{tab:main} lie well inside these bounds.
Heuristically $q_h(k)$ behaves like $(h/(h-1))^k$
up to a factor polynomial in $k$, so that the three-colour bounds of this method
grow roughly like $3^k$. We have not proved this.

\section{Computations}\label{sec:computations}

We checked every hypothesis behind the certificates deterministically, with
no sampling. The programs are custom C compiled with GCC~14.2 at \texttt{-O2} or
\texttt{-O3} and parallelised with OpenMP or POSIX threads, with independent
reimplementations as cross-checks. The machines used were an 8-vCPU Intel Xeon
virtual machine with 16~GB of RAM running Debian Linux (below, the Linux
machine) and an Apple Mac Studio with an M4 Max chip (16 cores and 64~GB of RAM)
running macOS~26 with Apple clang~17.

The window criterion of Lemma~\ref{lem:window} for the power-residue blocks was
evaluated by the scanners \texttt{scanfast.c} and \texttt{scanr.c}. Each cubic
($h=3$) block was then rechecked by direct enumeration. For $q\le 576{,}727$ the
verifier \texttt{exhaust\_omp.c} enumerated every pair $(a,d)$ with $1\le d\le q-1$, that is all
$q(q-1)$ progressions. For $q=576{,}727$ this means $332{,}613{,}455{,}802$
progressions, which took $840$~s on the Linux machine. For
$q=1{,}154{,}653$ and $q=1{,}376{,}317$ the enumeration ran over $d\le(q-1)/2$,
which is complete because reversing a progression replaces $d$ by $-d$ without
changing the colours met. These runs covered $666{,}611{,}197{,}878$ and
$947{,}123{,}554{,}086$ progressions and used the sharded variant
\texttt{exhaust\_rng.c} as eight shards on the Mac Studio.
None found a bad progression.

The $32$ distinct $h=4,5,6$ blocks, including the pigeonhole blocks, were checked
by the direct enumerator \texttt{exhr.c}, which computes the index by walking the
powers of a primitive root $g$.
The arc palettes were found with \texttt{arcscan.c}, which tests conditions (i)
and (ii) of Lemma~\ref{lem:arc}, and each of the four $(5,3)$ palettes used
was then checked against $(\star)$ directly by \texttt{exharc.c}, which
enumerates all $q(q-1)/2$ pairs $(x,d)$ with $1\le d\le(q-1)/2$ and computes
the index by walking powers of $g$ (under one second per palette on the Linux
machine), and again by an independent pure Python check \texttt{arcpy.py} over
all $q(q-1)$ pairs. As negative controls the same palettes fail at $k-1$, and the
palettes at $q=1931$ and $q=9791$ fail with the other unit $a$.

The two-colour bases for $k=9,11,12$ were brute-forced in full by an
independently written bitset checker \texttt{chk.c} (POSIX threads), which checked
$18{,}543{,}870{,}661$ progressions for $k=12$ in $274$~s, $1{,}880{,}558{,}604$
for $k=11$ and $106{,}404{,}382$ for $k=9$. For $k\ge13$ the certificate checker
\texttt{vdw\_cert.c} verified the criterion of Lemma~\ref{lem:legendre} for the
Legendre bases, taking from a few seconds to under a minute per prime for
$k=13,\dots,28$. The zipper search used \texttt{zipscan.c} and
\texttt{zipscan2.c}; every zipped prime passing for $k=9$ below $3000$ and for
$k=10$ below $6000$, $197$ positive controls in all, was fully brute-forced,
together with the negative control $p=2591$. The cyclic $h=2$ blocks of
Lemma~\ref{lem:zipper} were checked exhaustively by \texttt{exhf.c}.

For end-to-end checks an independent generator \texttt{gen.py} (pure Python~3)
built complete three-colourings. The colouring behind $W(3,9)>2{,}629{,}961$
(an older, Legendre base) was tested by \texttt{chk.c} and by a second
independent verifier \texttt{vindep.c}, and the colouring behind the table
entry $W(3,10)>10{,}036{,}978$ by \texttt{chk.c}
($5{,}596{,}713{,}168{,}872$ progressions). Stacked constructions of every
recurrence type ($h=2,3,4$ and $r=2,\dots,6$) at $k=6$ were brute-forced in full,
the largest being six colours over length $676{,}235$, i.e.\
$45{,}729{,}039{,}405$ progressions. A planted negative control was detected, as
were the negative controls $q=103,109$ at $k=9$, which fail. The complete
three-colouring of length $4{,}002{,}705$ behind the new $k=9$ record
(Theorem~\ref{thm:palette} with the cubic block $q=97$ and the zipped base of
length $41{,}265$) was brute-forced by \texttt{chk.c}, which went through $1{,}001{,}350{,}955{,}962$
progressions in $258$~s and found none monochromatic.

A search with PySAT (version 1.9.dev16) and Glucose~4 over all
three-colourings of $\Z_{99}$ for $k=9$ returned UNSAT, so $\Z_{99}$ has no
$9$-polychromatic $3$-colouring.

\section{Related work}\label{sec:related}

Polychromatic and panchromatic colourings also appear in nearby settings.
Axenovich, Goldwasser, Lidick\'y, Martin, Offner, Talbot and
Young~\cite{Axen} study polychromatic colourings for translates of a fixed set,
not progressions. Akhmejanova and Longbrake~\cite{AL} study panchromatic and
$(a{:}b)$-colourings of hypergraphs, and our blocks are panchromatic colourings of
the cyclic $k$-AP hypergraph. Liber~\cite{Lib} studies the cyclic van der
Waerden hypergraphs. Xiu, Li, Liang and Xu~\cite{XLLX} introduced set-colouring
van der Waerden numbers, in which each integer receives a set of colours as in
Theorem~\ref{thm:palette}. We have only seen the abstract of that paper.

\section{Code and data availability}

All source code, block parameters, base primes, logs and the record table are in
the accompanying archive, deposited at Zenodo with
doi:10.5281/zenodo.23489549 (\url{https://doi.org/10.5281/zenodo.23489549}). The archive
includes \texttt{records.json} and \texttt{blocks.json}, and every certificate can be
regenerated and checked with the included programs.

\section*{Acknowledgements}

The two-colour base primes for $k\ge 13$ are those of Monroe~\cite{Mon} and
Pouly~\cite{Pou,PouAtlas}.

\end{document}